\documentclass[letterpaper, 10 pt, conference]{ieeeconf}  

\IEEEoverridecommandlockouts                              
\usepackage{amsmath}		
\usepackage{amscd}
\usepackage{amsfonts}
\usepackage{amssymb}		
\usepackage{mathtools}		

\usepackage[english]{babel}		
\usepackage[utf8]{inputenc}	
\usepackage{calc}
\usepackage{thumbpdf}			
\usepackage{latexsym}
\usepackage{xcolor}
\usepackage{enumerate}		
\usepackage{url}				
\usepackage{cite}	
\usepackage{graphics} 
\usepackage{epsfig} 
\usepackage{color}

\def\trace{{\rm Tr\,}}

\DeclareMathOperator*{\argmin}{arg\,min}

\newcommand{\R}{\mathbb{R}}
\newcommand{\N}{\mathbb{N}}

\newcommand{\Normal}{\mathcal{N}}
\newcommand{\F}{\mathcal{F}}

\newcommand{\Prob}{\mathbb{P}}
\newcommand{\E}{\mathbb{E}}

\newcommand{\RR}{\mathcal{R}}
\newcommand{\Y}{\mathcal{Y}}

\usepackage[bookmarks,bookmarksnumbered,colorlinks]{hyperref}
\hypersetup{linkcolor = black,anchorcolor = black,citecolor =
	black,filecolor = black,urlcolor = black}

\usepackage{algorithm}
\usepackage[noend]{algpseudocode}

\newtheorem{defn}{Definition}[section]
\newtheorem{lem}[defn]{Lemma}
\newtheorem{thm}[defn]{Theorem}
\newtheorem{cor}[defn]{Corollary}

\newtheorem{ass}[defn]{Assumption}
\newtheorem{rem}[defn]{Remark}

\definecolor{darkgreen}{rgb}{0,0.6,0}

\title{Closed-loop analysis of linear stochastic MPC with risk-averse constraints}
\author{Jonas Schießl, Ruchuan Ou, Michael H. Baumann, Timm Faulwasser, and Lars Grüne
\thanks{
The authors gratefully acknowledge that this work was funded by the Deutsche Forschungsgemeinschaft (DFG, German Research Foundation) – project number 499435839.}
\thanks{Jonas Schießl, Michael H. Baumann and Lars Grüne are with Mathematical Institute, University of Bayreuth, Germany,
        \tt\small \{jonas.schiessl,michael.baumann,lars.gruene\} @uni-bayreuth.de}%
\thanks{Ruchuan Ou and Timm Faulwasser are with the Institute of Control Systems, Hamburg University of Technology, Hamburg, Germany,
        \tt\small ruchuan.ou@tuhh.de, timm.faulwasser@ieee.org}%
}

\begin{document}

\maketitle

\pagestyle{empty}

\begin{abstract}
    Chance constraints are widely used in stochastic model predictive control (MPC) to enforce probabilistic state and input constraints in the presence of unbounded disturbances. 
    However, they only restrict violation probabilities and do not account for the magnitude of rare but severe constraint violations.
    In this paper, we extend the  indirect feedback approach for linear stochastic MPC from chance constraints to risk-averse constraints like the conditional value-at-risk. 
    For the resulting risk-averse MPC scheme, we establish recursive feasibility and closed-loop constraint satisfaction.
    Furthermore, based on a stochastic dissipativity notion and suitable conditions on the terminal ingredients we show that (near)-optimality of the averaged closed-loop performance can be ensured.
\end{abstract}

\section{INTRODUCTION}
The structured consideration of system uncertainty (either induced by plant-model mismatch or stemming from exogenous disturbances) is crucial in many control contexts. When stochastic optimal control or predictive control is considered, the consideration of chance constraints has become a standard tool, see, e.g., \cite{Hewing2018,Hewing2020,kohler2025predictive,schluter2023stochastic,farina2016stochastic,paulson2020stochastic,bavdekar2016stochastic,schluter2022stochastic}. 
However, chance constraints do, in general,  not allow to avoid rare outcomes with bad performance. Risk measures, on the other hand, are well suited to avoiding rare outcomes with bad performance~\cite{rockafellar2007coherent,shapiro2021lectures}. Examples include conditional value-at-risk, entropic value-at-risk and others. 
In previous work we analyzed stochastic MPC with risk-averse objectives \cite{schiessl2025towards}. Moreover, \cite{yin2023risk,dixit2021risk} suggest the consideration of  risk-averse constraint formulations in stochastic MPC. Yet, to the best of our knowledge, the formal closed-loop analysis of  stochastic MPC appears to be mostly limited to chance-constrained formulations,  cf.\ \cite{Hewing2020,kohler2025predictive,schluter2023stochastic,schluter2022stochastic}.

On this canvas, this paper makes first steps towards developing an analysis framework for stochastic MPC of linear systems subject to potentially non-Gaussian disturbances and considering  generic risk-averse constraint formulations. In particular, we extend the indirect feedback approach presented in \cite{Hewing2018,Hewing2020} to the consideration of risk-averse constraints using risk measures.  Furthermore, based on dissipativity notions for stochastic systems introduced in \cite{Schiessl2025,Schiessl2023a,Schiessl2025b}, we provide a rigorous  analysis of the averaged performance of the closed MPC loop for not necessarily quadratic cost functions. In contrast to \cite{Hewing2018}, we provide a lower bound on the averaged performance defined by a stationary solution \emph{and} we derive an upper bound which holds for general stage costs if the terminal ingredients are suitably chosen. The core contributions of the paper are twofold:
 (i) We extend the indirect feedback approach of \cite{Hewing2018,Hewing2020} to risk-averse constraint formulations using risk measures and to non-quadratic stage costs.
(ii) Using stochastic dissipativity concepts, we derive a novel lower bound on the averaged performance of stochastic linear MPC whereby we do not require Gaussianity of the disturbance distribution.

The remainder of this paper is structured as follows: Section \ref{sec:problem} introduces the setting and problem formulation, while Section \ref{sec:sMPC} recalls the indirect feedback approach by \cite{Hewing2018,Hewing2020}. Section \ref{sec:CLGuarantees} presents our main findings, while in Section \ref{sec:GaussianSetting} we focus on the special case of Gaussian uncertainty and quadratic stage costs. Section \ref{sec:Numerics} draws upon a numerical example to illustrate our findings, while the paper ends with conclusions in Section \ref{sec:Conclusion}.

\section{PROBLEM FORMULATION} \label{sec:problem}

Let $A \in \R^{n \times n}$, $B \in \R^{n \times l}$, such that the pair $(A,B)$ is controllable.
Then, for an \emph{i.i.d.} sequence $W(0),W(1),\ldots$ such that $W(k)$ is independent of $X(k)$ and $U(k)$ for all $k \in \N_0$, we consider linear stochastic systems of the form

\begin{equation} \label{eq:StochSystem}
  X(k+1) = AX(k) + BU(k) + W(k), \quad X(0) = X_0.
\end{equation}
Here, the initial condition $X_0$, the states  $X(k)$, the controls $U(k)$, and the noise $W(k)$ are considered to be random variables on the probability space  $(\Omega, \mathcal{F}, \Prob)$, i.e., 
$X(k) \in \RR(\Omega,\R^n)$, $U(k) \in \RR(\Omega,\R^l)$, and $W(k) \in \RR(\Omega,\R^m)$ with 
\[\RR(\Omega,\Y) := \{X: (\Omega,\mathcal{F},\Prob) \rightarrow (\Y,\mathcal{B}(\Y)) \mbox{ measurable} \}, \]
for $\Y=\R^n$, $\R^l$, or $\R^m$, where $\mathcal{B}(\Y)$ denotes the Borel $\sigma$-algebra on $\Y$.
Furthermore, we assume that the control sequence $\mathbf{U} = (U(0),U(1),\ldots)$ is adapted to the stochastic filtration $(\F_k)_{k \in \N_0}$ defined by
\begin{equation} \label{eq:Filtration}
    \F_k = \sigma(X(0),\ldots,X(k)), ~ \text{for all} ~ k \in \N_0.
\end{equation}
The last condition can be seen as a a causality requirement, which guarantees that we only take past and present but not future events into account for our control design. Moreover, note that the setting above allows the disturbance $W(k)$ to be non-Gaussian.

To extend system \eqref{eq:StochSystem} to an optimal control problem we consider stage costs in expectation of the form
\begin{equation} \label{eq:stagecost}
    \ell(X,U) = \E[g(X,U)],
\end{equation}
where the deterministic stage costs $g: \R^n \times \R^l \to \R$ is a continuous function bounded from below.
Additionally, we impose linear risk-averse constraints of the form 
\begin{equation} \label{eq:constraints}
    \begin{split}
        \rho(c_i^\top X) &\leq p_i, \quad i=1,\ldots,m_x \\
        \rho(d_i^\top U) &\leq q_i \quad i=1,\ldots,m_u.
    \end{split}
\end{equation}
Here, the mapping $\rho(Y)$ for $Y \in \RR(\Omega,\R)$ is a risk measure in the sense of the following definition.

\begin{defn} \label{defn:risk}
    A mapping $\rho: \RR(\Omega,\R) \to \R \cup \{+\infty\}$ is called \emph{risk measure} if it is 
    \begin{enumerate}[(i)]
        \item \emph{translative}, i.e., $\rho(Y+C) = \rho(Y) + c$ for all $Y \in \RR(\Omega,\R)$ and $c \in \R$. 
        \item \emph{monotone}, i.e.,  $\rho(Y_1) \geq \rho(Y_2)$ for all  $Y_1,Y_2 \in \RR(\Omega,\R)$ with $Y_1 \geq Y_2$ almost surely.\\[-4mm]
    \end{enumerate}
\end{defn}

Furthermore, we assume that the risk measure is law-invariant, i.e., $\rho(Y) = \rho(Z)$ for all $Y \sim Z$.

Commonly used law-invariant risk measures are, e.g., the value-at-risk  \eqref{eq:VaR} or the conditional value-at-risk \eqref{eq:CVaR}, which we will discuss in Section \ref{sec:GaussianSetting}.

Then, the stochastic optimal control problem with horizon $N \in \N \cup \{\infty\}$  reads
\begin{equation} \label{eq:stochOCP}
    \begin{split}
        \min_{\mathbf{U}} ~J_N(X_0,\mathbf{U}) &:= \sum_{k=0}^{N-1} \ell(X(k),U(k)) \\
        s.t. ~ X(k+1) &= AX(k) + BU(k) + W(k)\\ 
        X(0) &= X_0, \quad \sigma(U(k)) \subseteq \F_k, \\
        \rho(c_i^\top X(k)) &\leq p_i, \quad i=1,\ldots,m_x \\
        \rho(d_i^\top U(k)) &\leq q_i \quad i=1,\ldots,m_u
    \end{split}
\end{equation}
for which we want to approximate a solution on the infinite-horizon that satisfies the risk-averse constraints for all times.

\section{INDIRECT-FEEDBACK STOCHASTIC MPC} \label{sec:sMPC}
To calculate an approximation of the solution to \eqref{eq:stochOCP} for $N = \infty$ we use an indirect-feedback stochastic MPC scheme, cf.\ \cite{Hewing2018,Hewing2020}. The idea of the indirect-feedback approach is to use a deterministic  prediction $z(k) \in \R^n$ for evaluation of tightened constraints in open loop while the optimization of the cost is performed subject to the most recent state measurement.

To this end, we use a linear-affine feedback parametrization of the control during the open-loop optimization, i.e., in \eqref{eq:stochOCP} it holds that $U(k) = K X(k) + v_k$ 
for all $k \in \{0,\ldots,N-1\}$, where $K \in \R^{l \times n}$ is a fixed linear feedback-gain stabilizing the pair $(A,B)$ and $v_k \in \mathbb{V} \subseteq \R^l$ is the free control variable.
Then, by defining the prediction $z(k) \in \R^n$
\begin{equation} \label{eq:dynamicsz}
\begin{split}
    z(k+1) &= (A+BK) z(k) + Bv_k + \E[W(k)], \\
    z(0) = \E[X_0]
\end{split}
\end{equation}
the full state can be written as $X(k) = z(k) + E(k)$ with
\begin{equation} \label{eq:dynamicsE}
\begin{split}
    E(k+1) &= (A+BK)E(k) + W(k) - \E[W(k)], \\
    E(0) &= X_0-z_0.
\end{split}
\end{equation}
Since for a given $z_0 \in \R^n$ the dynamics~\eqref{eq:dynamicsz} are deterministic, we obtain
\begin{equation}
    \begin{split}
        \rho(c_i^\top X(k)) &= \rho(c_i^\top(z(k) + E(k))) \\
        &= c_i^\top z(k) + \rho(c_i^\top E(k)).
    \end{split}
\end{equation}
and 
\begin{equation}
\begin{split}
    \rho(d_i^\top U(k)) &= \rho(d_i^\top (K(z(k) + E(k))) + v_k) \\
    &= d_i^\top (Kz(k) + v_k) + \rho(d_i^\top KE(k)).
\end{split}
\end{equation}
Hence we can rewrite the risk-averse constraints as 
\begin{equation}
\begin{split}
     c_i^\top z(k) &\leq p_i - \rho(c_i^\top E(k)), ~ i=1,\ldots,m_x \\
     d_i^\top (Kz(k) + v_k) &\leq q_i - \rho(d_i^\top KE(k)), ~ i=1,\ldots,m_u.
\end{split}
\end{equation}
Note that since we fixed the feedback matrix $K$, the dynamics of $E$ from \eqref{eq:dynamicsE} do not depend on the control input and thus the evolution of $E$ is not affect by the optimization.
Using the feedback parametrization $U(k) = KX(k) + v_k$ and prediction $z(k)$ the resulting open-loop problem on horizon $N \in \N$ with initial values $x_j \in \R^n$, $z_j \in \R^n$, $E_j \in \RR(\Omega,\R^n)$ reads
\begin{equation} \label{eq:openloopIfSMPC}
    \begin{aligned}
        \min_{\mathbf{v}} \sum_{k=0}^{N-1} &\ell(X(k),U(k)) + F(X(N)) \\
        s.t.~X(k+1) &= (A+BK) X(k) + Bv_k + W(k) \\
        z(k+1) &= (A+BK)z(k) + Bv_k + \E[W(k)] \\
        E(k+1) &= (A+BK) E(k) + W(k) \\
        U(k) &= KX(k) + v_k, ~ v_k \in \mathbb{V}, ~ z(N) \in \mathbb{Z}_f \\
        X(0) &= x_j, ~ z(0) = z_j, E(0) = E_j\\
        c_i^\top z(k) &\leq p_i - \rho(c_i^\top E(k)),  i=1,\ldots,m_x \\
        d_i^\top (Kz(k) + v_k) &\leq q_i - \rho(d_i^\top KE(k)),  i=1,\ldots,m_x
    \end{aligned}
\end{equation}
and by 
\begin{align*}
    \mathcal{V}_N(x_j,z_j,E_j) &= \inf_{\mathbf{v}} \sum_{k=0}^{N-1} \ell(X(k),U(k)) + F(X(N)) \\
    &s.t.~\text{constraints from~\eqref{eq:openloopIfSMPC}}
\end{align*}
we denote the optimal value function on horizon $N \in \N$ corresponding to this problem.
Note that in contrast to \eqref{eq:stochOCP}, in problem we 
added terminal ingredients in~\eqref{eq:openloopIfSMPC}, namely the terminal set $\mathbb{Z}_f$ and the terminal penalty $F(X) = \E[g_f(X)]$ with $g_f: \R^n \to \R$. 
Such terminal ingredients are common in MPC to ensure recursive feasibility and stability, cf.\ \cite{mayne2000constrained,grune2017nonlinear}, and will also be used to derive our closed-loop guarantees in Section~\ref{sec:CLGuarantees}.
The resulting indirect feedback SMPC scheme is summarized in Algorithm~\ref{alg:ifStochMPC}. 

\begin{algorithm}
    \caption{Indirect feedback SMPC}
    \label{alg:ifStochMPC}
    \begin{algorithmic}
    \State \textbf{Input:} Fixed stabilizing feedback $K \in \R^{l \times n}$, feasible initial state $X_0$.
        \State Measure the state $x_0 = X_0(\omega)$, calculate $z_0 = \E[X^{cl}(0)]$ and set $X^{cl}(0,\omega)=x_j$, $Z^{cl}(0,\omega) = z_0$, $E_0 = X_0 - z_0$.
        \For{$j=0,1,\ldots$}\vspace{2mm}
            \State 1.) Solve the stochastic optimal control problem \eqref{eq:openloopIfSMPC} and obtain the solution $\mathbf{v}^* := (v_0^*,\ldots,v_{N-1}^*).$\vspace{2mm}
            \State 2.) Compute $E_{j+1} = (A+BK)E_j + W(j),$ predict 
            $z_{j+1} = (A+BK)z_j + Bv_0^* + \E[W(j)],$
            and set $Z^{cl}(j+1,\omega) = z_{j+1}$.\vspace{2mm}
            \State 3.) Set $V^{cl}(j,\omega) = v_0^*$, apply the feedback 
            $U^{cl}(j,\omega) = Kx_j + v_0^*$ 
            to system \eqref{eq:StochSystem} and measure the next state $x_{j+1} = X^{cl}(j+1,\omega)$.
        \EndFor
    \end{algorithmic}
\end{algorithm}

Note that due to the initializations at time $j=0$ we get
\begin{align*}
    Z^{cl}(0) &= z_0 = \E[X_0] = \E[X^{cl}(0)], \\
    X^{cl}(0) &= Z^{cl}(0) + E_0 = z_0 + E_0.
\end{align*}
However, while for times $j \geq 1$ it still holds that 
\begin{equation}
    X^{cl}(j) = Z^{cl}(j) + E_j
\end{equation}
in general $Z^{cl}(j) = \E[X^{cl}(j)]$ would not hold anymore but only
\begin{equation}
    Z^{cl}(j) = \E[X^{cl}(j) \mid \F_{j-1}]
\end{equation}
since the control value $U^{cl}(j,\omega) = v_0^*$ in Algorithm~\ref{alg:ifStochMPC} depends on the current measurements through optimization.
This particularly emphasizes that $Z^{cl}(j)$ is a random variable and that $Z^{cl}(j)$ as well as the closed-loop controls are depending on the whole history of states, i.e., $Z^{cl}(j+1)$ and $U^{cl}(j)$ are $F_j$-measurable.

\begin{rem} \label{rem:constraintTightening}
    Note that the dynamics of $E$ from~\eqref{eq:dynamicsE} do not depend on $v$ and hence are independent of the optimization.
    Therefore the sequences $\rho(c_i^\top E(j))$ and $\rho(d_i^\top KE(j))$, which are necessary for constraint evaluation, can be computed offline in advance.

    However, usually it is rather difficult to evaluate $\rho(c_i^\top E(j))$ and $\rho(d_i^\top KE(j))$ exactly unless we consider special cases as in Section~\ref{sec:GaussianSetting}. 
    One possibility to get at least an approximation of these terms is for example to use a Monte-Carlo sampling.
    Moreover, one could also further tighten the constraints if there exists sequences $\tilde{c}(j)$ and $\tilde{d}(j)$ such that
    \begin{equation}
         \rho(c_i^\top E(j)) \leq \tilde{c}(j), \qquad \rho(d_i^\top KE(j)) \leq \tilde{d}(j)
    \end{equation}
    holds.
    If such sequences are known, we can simply replace the terms $\rho(c_i^\top E(k))$ and $\rho(d_i^\top KE(k))$ in problem~\eqref{eq:openloopIfSMPC} by $\tilde{c}(j)$ and $\tilde{d}(j)$.
    While this of course would lead to a more conservative formulation, the results of this paper still hold if the terminal set is constructed appropriately as explained after Theorem~\ref{thm:recursiveFeasibility}.
\end{rem}

\section{CLOSED-LOOP GUARANTEES} \label{sec:CLGuarantees}

In this section we aim to provide closed-loop guarantees for Algorithm~\ref{alg:ifStochMPC}, particularly showing closed-loop constraint satisfaction and averaged (near-)optimality.
Note that this algorithm does not use the simplification from the Gaussian setting from Section~\ref{sec:GaussianSetting} and thus, our closed-loop guarantees are theoretically guaranteed for arbitrary initial conditions and distributions.
Moreover, as we see in Section~\ref{sec:GaussianSetting} all the assumptions made in this section can be satisfied in the linear-quadratic case, which enables us to transfer the derived results to the computationally more tractable Algorithm~\ref{alg:ifStochMPCmoments}. 

\subsection{Recursive Feasibility}
Before we deal with constraint satisfaction and optimality estimates, we first show that Algorithm~\ref{alg:ifStochMPCmoments} is recursively feasible, i.e., if we start with an initial condition $X_0$ for which the problem \eqref{eq:openloopIfSMPC} can be solved, then it can be solved for all subsequent steps of the MPC loop.

To this end, we make the following assumption, which is akin to \cite[Assumption~1]{Hewing2020}.

\begin{ass}  \label{ass:terminalSet}
    \begin{enumerate}[(i)]
        \item There exists $\mathbb{Z}_f \neq \emptyset$ such that  
        \begin{equation}
            c_i^\top  z \leq p_i - \sup_{j \in \N_0} \rho(c_i^\top E(j)).
        \end{equation}
        holds for all $z \in \mathbb{Z}_f$.
        \item There exists $v_f \in \mathbb{V}$ such that
        \begin{align*}
            (A+BK)z + Bv_f + \E[W(k)] \in \mathbb{Z}_f, \\
            d_i^\top (Kz + v_f) \leq q_i - \sup_{j \in \N_0} \rho(d_i^\top KE(j))
        \end{align*}
        holds for all $z \in \mathbb{Z}_f$.
    \end{enumerate}
\end{ass}

Using this assumption, we can establish recursive feasibility of Algorithm~\ref{alg:ifStochMPCmoments} in an analogous way to \cite[Theorem~1]{Hewing2020}.

\begin{thm} \label{thm:recursiveFeasibility}
    If problem~\eqref{eq:openloopIfSMPC} is feasible for the initial condition $z_0 = \E[X_0]$, then Algorithm~\ref{alg:ifStochMPC} is recursive feasible, i.e., feasible for all times $j \in \N_0$.
\end{thm}
\begin{proof}
    Let $\mathbf{v}^* = (v_0^*,\ldots,v_{N-1}^*)$ be an optimal solution of problem~\eqref{eq:openloopIfSMPC} at time $j \in \N_0$. 
    Then, by Assumption~\ref{ass:terminalCost} the sequence  
    \begin{equation}
         \tilde{\mathbf{v}} := (v_1^*,\ldots,v_{N-1}^*,v_f)
    \end{equation}
    satisfies the constraints in \eqref{eq:openloopIfSMPC} for
    $$z_{j+1} = (A+BK) z_j + Bv_0^* + E[W(j)]$$ 
    and all $x_j \in \R^n$.
    Hence $\tilde{\mathbf{v}}$ is an admissible control sequence for time $j+1$, which proves the claim since problem~\eqref{eq:openloopIfSMPC} is feasible at time $j=0$ by assumption.
\end{proof}

Note that since the feedback $K\in \R^{l \times n}$ stabilizes the pair $(A,B)$ there exists a distribution $P^s_E$ such that $E(j)$ converges in distribution to $P^s_E$ for suitable $E_0$, i.e, $E(j) \xrightarrow{d} P_E^s$ for $j \to \infty$.
Hence, the suprema in Assumption~\ref{ass:terminalSet} exist if the initial value $E_0$ is not degenerated, since the risk measure $r$ is assumed to be law-invariant. 

Furthermore, if one uses an upper bound on the risk as explained in Remark~\ref{rem:constraintTightening} we must also consider this in the construction of the terminal set $\mathbb{Z}_f$ by replacing $\rho(c_i^\top E(j))$ with $\tilde{c}(j)$ and $\rho(d_i^\top KE(j))$ with $\tilde{d}(j)$ respectively. 

\subsection{Constraint Satisfaction}

Since in closed loop the value $z_j$ represents $\E[X^{cl}(j) \mid \F_{j-1}]$ rather than the unconditioned expectation $\E[X^{cl}(j)]$ it is not obvious that the proposed risk-averse constraints \eqref{eq:constraints} are satisfied in closed-loop.
The following theorem shows that the considered restrictions in open loop are indeed sufficient to obtain closed-loop constraint satisfaction.

\begin{thm} \label{thm:constraints}
    The risk-averse constraints \eqref{eq:constraints} are satisfied in closed loop, i.e.,
    \begin{equation*}
        \begin{split}
            \rho(c_i^\top  X^{cl}(j)) &\leq p_i, \quad i=1,\ldots,m_x \\
            \rho(d_i^\top  U^{cl}(j)) &\leq q_i \quad i=1,\ldots,m_u,
        \end{split}
    \end{equation*}
    holds for all times $j \in \N_0$, where $X^{cl}(j)$ and $U^{cl}(j)$ are generated pointwisely according to Algorithm~\ref{alg:ifStochMPCmoments}.
\end{thm}
\begin{proof}
    For the closed-loop states and controls from Algorithm~\ref{alg:ifStochMPC} it holds that    
    \begin{align*}
        X^{cl}(j) &= Z^{cl}(j) + E_j \\
        U^{cl}(j) &= KX^{cl}(j) + V^{cl}(j) \\
        &= K (Z^{cl}(j) + E_j) + V^{cl}(j).
    \end{align*}
    
    Furthermore, due to the proposed risk-averse constraints in the open-loop problem~\eqref{eq:openloopIfSMPC} we can conclude that 
    \begin{align}
        c_i^\top  Z^{cl}(j) &\leq p_i - \rho(c_i^\top  E_j) \\
        d_i^\top (KZ^{cl}(j) + V^{cl}(j)) &\leq q_i - \rho(d_i^\top KE_j)
    \end{align}
    holds almost surely.

    Hence, we get by monotonicity and translativity of the risk measure, cf.\ Definition~\ref{defn:risk}, that
    \begin{align*}
        \rho(c_i^\top  X^{cl}(j)) &= \rho(c_i^\top  (Z^{cl}(j) + E_j) \\
        &\leq p_i - \rho(c_i^\top  E_j) + \rho(c_i^\top  E_j) = p_i
    \end{align*}
    and 
    \begin{align*}
        \rho(d_i^\top  U^{cl}(j)) &= \rho(d_i^\top  K (Z^{cl}(j) + E_j) + V^{cl}(j)) \\
        &\leq q_i - \rho(d_i^\top  K E_j) + \rho(d_i^\top  K E_j) = q_i
    \end{align*}
    holds for all $j \in \N_0$.
\end{proof}

Note that the proof for the constraint satisfaction relies on the fact that we can split the closed-loop state into a stochastic part $E(j)$, which is independent of the control, and a nominal part $z(j)$, which we can restrict almost surely in a suitable way.
Hence, we conjecture that our results can also be obtained for different splittings and parametrizations of the control which leads to a splitting with the same properties.

\subsection{Averaged Performance Optimality} \label{sec:Performance}
As the final part of our closed-loop analysis we will give optimality estimates for the averaged performance. 
These findings will be based on a stochastic dissipativity notion developed in \cite{Schiessl2025}.
There it was shown that in contrast to the deterministic setting (strict) dissipativity notions can be formulated on different layers, such as moments, distributions or random variables. 
However, in the following we will use the notion formulated with respect to random variables, which leads to the following stationarity concept.

\begin{defn}
    A pair of stochastic processes $(\mathbf{X}^s,\mathbf{U}^s)$ given by
    \begin{equation} \label{eq:sys_stat}
        X^s(k+1) = AX^s(k) + BU^s(k) + EW(k)
    \end{equation}
    with $U^s(k) = \pi^s(X^s(k))$ is called stationary for system \eqref{eq:StochSystem} if $X^s(k)$ and $U^s(k)$ satisfy the constraints \eqref{eq:constraints} for all times $k \in \N_0$
    and there exist probability distributions $P^s_X$, $P^s_U$, and $P^s_{X,U}$ with
    \begin{equation*}
    \begin{split}
        X^s(k) \sim P^s_X, \quad U(k) \sim P^s_U, \quad (X^s(k),U^s(k)) \sim P^s_{X,U}
    \end{split}
    \end{equation*}
    for all $k \in \N_0$. 
\end{defn}

Using this definition of a stationary process as the replacement of the deterministic steady state, we can define stochastic dissipativity in the following way, where we denote by $\ell(\mathbf{X}^s,\mathbf{U}^s)$ the stage costs of the stationary pair, which are independent of $k$ due to the stationarity of the distributions.

\begin{defn} \label{defn:stochDissi}
    Consider a pair of stationary stochastic processes $(\mathbf{X}^s,\mathbf{U}^s)$ with $\vert \ell(\mathbf{X}^s,\mathbf{U}^s) \vert < \infty$.
    Then, we call the stochastic optimal control problem \eqref{eq:stochOCP} stochastically dissipative at $(\mathbf{X}^s,\mathbf{U}^s)$, if there exists a law-invariant storage function $\lambda: \RR(\Omega,\R^n) \to \R$ bounded from below such that 
    \begin{equation} \label{eq:DissiIneq}
    \begin{split}
        \ell(X(k),&U(k)) - \ell(\mathbf{X}^s,\mathbf{U}^s)
        + \lambda(X) \\
        &- \lambda(AX(k)+ BU(k) + W(k)) \geq 0
    \end{split}
    \end{equation}
    holds for all $k \in \N_0$ and all $X(k)$, $U(k)$ satisfying \eqref{eq:Filtration} and the constraints \eqref{eq:constraints}.
\end{defn}

Based on stochastic dissipativity we can now establish a lower bound on the asymptotic averaged performance for all admissible control sequences $\mathbf{U}$ of problem \eqref{eq:stochOCP}. 
Here, for a given control sequence $\mathbf{U}$ and initial state $X_0$ we denote by $X_{\mathbf{U}}(k,X_0)$ the solution to \eqref{eq:StochSystem} at time $k$.

\begin{thm} \label{thm:lowerBoundPerformance}
    Assume that the stochastic optimal control problem \eqref{eq:stochOCP} is dissipative at $(\mathbf{X}^s,\mathbf{U}^s)$.
    Then, it holds that 
    \begin{equation}
        \liminf_{L \to \infty} \frac{1}{L} \sum_{k=0}^{L-1} \ell(X_{\mathbf{U}}(k,X_0),U(k)) \geq \ell(\mathbf{X}^s,\mathbf{U}^s)
    \end{equation}
    for all $\mathbf{U}$ and $X_0$ such that $U(k)$ and $X_{\mathbf{U}}(k,X_0)$ satisfy the constraints \eqref{eq:constraints} and the filtration condition \eqref{eq:Filtration} for all $k \in \N_0$.
\end{thm}
\begin{proof}
    By dissipativity we know that there exists a uniform lower bound $-C^l_{\lambda} < 0$ on $\lambda$ such that
    \begin{equation*}
    \begin{split}
        \frac{1}{L} &\sum_{k=0}^{L-1} \ell(X(k), U(k)) \\
        &\geq \frac{1}{L} \sum_{k=0}^{L-1} \ell(\mathbf{X}^s,\mathbf{U}^s) - \lambda(X_{\mathbf{U}}(k)) + \lambda(X(k+1)) \\
        &\geq \ell(\mathbf{X}^s,\mathbf{U}^s) - \frac{\lambda(X_0)}{L} - \frac{C^l_{\lambda}}{L},
    \end{split}
    \end{equation*}
    which proves the claim by letting $K$ go to infinity.
\end{proof}

Note that the lower bound from Theorem~\ref{thm:lowerBoundPerformance} also holds for the closed-loop solution $(\mathbf{X}^{cl},\mathbf{U}^{cl})$ since it satisfies the constraints due to Theorem~\ref{thm:constraints}.

Next we will show that we can also bound the closed-loop performance from above given suitable terminal ingredients as defined in the following assumption.

\begin{ass} \label{ass:terminalCost}
    There exists a stationary pair $(\mathbf{X}^s,\mathbf{U}^s)$ and a constant $C_f \geq 0$ such that for all $X \in \RR(\Omega,\R^n)$ and $v_f$ from Assumption~\ref{ass:terminalSet} the inequality
    \begin{equation}
    \begin{split}
        &F(AX + B(KX+v_f) + W(N)) \\
        & \quad \leq F(X) - \ell(X,KX+v_f)) + \ell(\mathbf{X}^s,\mathbf{U}^s) + C_f
    \end{split} 
    \end{equation}
    holds.
\end{ass}

The following theorem introduces the upper bound on the asymptotic averaged performance based on this assumption.

\begin{thm} \label{thm:upperBoundPerformance}
    Let Assumption~\ref{ass:terminalCost} hold.
    Then, it holds that
    \begin{equation}
        \limsup_{L \to \infty} \frac{1}{L} \sum_{k=0}^{L-1} \ell(X^{cl}(k),U^{cl}(k)) \leq \ell(\mathbf{X}^s,\mathbf{U}^s) + C_f.
    \end{equation}
\end{thm}
\begin{proof}
    Consider a given measurement $x_j = X^{cl}(j,\omega) \in \R^n$, prediction $z_j = Z^{cl}(j,\omega)$, and $E_j \in \RR(\Omega,\R^n)$, and assume that
    \begin{equation}
        \mathbf{v}^* \in \argmin \mathcal{V}_N(X^{cl}(j,\omega),Z^{cl}(j,\omega),E_j)
    \end{equation}
    holds, where $\mathcal{V}_N$ the optimal value function to problem~\eqref{eq:openloopIfSMPC}.
    Furthermore, set
    \begin{align*}
        X^*(k+1) &= AX^*(k) + BU^*(k) + W(k), \quad X^*(0) = x_j \\
        U^*(k) &= KX^*(k) + v^*_k.
    \end{align*}
    
    Since $\tilde{\mathbf{v}} = (v_1^*,\ldots,v_{N-1}^*,v_f)$ with $v_f \in \mathbb{V}$ from Assumption~\ref{ass:terminalSet} is admissible for time $j+1$, cf.\ Theorem~\ref{thm:recursiveFeasibility}, we then get
    \begin{align*}
        \mathcal{V}_N&(X^{cl}(j,\omega),Z^{cl}(j,\omega),E_j) \\
        &- \E \left[ \mathcal{V}_N(X^{cl}(j+1),Z^{cl}(j+1),E_{j+1}) \mid \F_j \right] (\omega) \\
        \leq& \sum_{k=0}^{N-1} \ell(X^*(k),U^*(k)) + F(X^*(N)) \\
        &- \bigg( \sum_{k=0}^{N-2} \ell(X^*(k+1),U^*(k))) \\
        & \qquad + \ell(X^*(N),KX^*(N)+v_f) \\
        & \qquad + F \left( AX(N) + B(KX(N)+v_f)) + W(N) \right) \bigg) \\
        =& \, \ell(X^{cl}(j,\omega),\pi_0(X^{cl}(j,\omega))) \\ 
        &- \ell(X(N),KX(N)+v_f)) + F(X(N)) \\
        &- F(AX(N) + B(KX(N)+v_f) + W(N)).
    \end{align*}
    Using Assumption~\ref{ass:terminalCost} this implies 
    \begin{align*}
        \mathcal{V}_N&(X^{cl}(j,\omega),Z^{cl}(j,\omega),E_j) \\
        &- \E \left[ \mathcal{V}_N(X^{cl}(j+1),Z^{cl}(j+1),E_{j+1}) \mid \F_j \right] (\omega) \\
        \leq& \ell(X^{cl}(j,\omega),\pi_0(X^{cl}(j,\omega))) - \ell(\mathbf{X}^s,\mathbf{U}^s) - C_f.
    \end{align*}
    
    Taking the expectation yields
    \begin{equation*}
    \begin{split}
        \ell(X^{cl}(j),U^{cl}(j)) &\leq \ell(\mathbf{X}^s,\mathbf{U}^s) + C_f \\
        & + \E[\mathcal{V}_N(X^{cl}(j),Z^{cl}(j),E_j)] \\
        &- \E[\mathcal{V}_N(X^{cl}(j+1),Z^{cl}(j+1),E_{j+1})]
    \end{split}
    \end{equation*}
    and thus, we get
    \begin{equation*}
    \begin{split}
        \limsup_{L \to \infty} &\frac{1}{L} \sum_{j=0}^{L-1} \ell(X^{cl}(j),U^{cl}(j)) \\
        \leq \limsup_{L \to \infty} &\frac{1}{K} \bigg(\E[\mathcal{V}_N(X^{cl}(j),Z^{cl}(j),E_j)] \\
        &\quad - \E[\mathcal{V}_N(X^{cl}(j+1),Z^{cl}(j+1),E_{j+1})] \bigg) \\
        &+ \ell(\mathbf{X}^s,\mathbf{U}^s) + C_f \\
        \leq \limsup_{L \to \infty}& \frac{1}{L} \E[\mathcal{V}_N(X_0,z_0,E_0)] - \frac{C_g}{L} + \ell(\mathbf{X}^s,\mathbf{U}^s) + C_f \\
        =& \ell(\mathbf{X}^s,\mathbf{U}^s) + C_f
    \end{split}
    \end{equation*}
    where $C_g \in \R$ is a lower bound on the deterministic stage costs $g$.
\end{proof}

To conclude the findings of this section, the following result combines Theorem~\ref{thm:lowerBoundPerformance} and Theorem~\ref{thm:upperBoundPerformance} to show (near-)optimality of closed-loop solutions in the averaged performance sense.

\begin{cor} \label{cor:performance}
    Let Assumptions~\ref{ass:terminalSet} and ~\ref{ass:terminalCost} hold and assume that the stochastic optimal control problem~\eqref{eq:stochOCP} is stochastically dissipative at $(\mathbf{X}^s,\mathbf{U}^s)$. 
    Then, the closed-loop solution from Algorithm~\ref{alg:ifStochMPC} has near-optimal averaged performance, i.e.,
    \begin{equation*}
    \begin{split}
        \ell(&\mathbf{X}^s,\mathbf{U}^s) \leq \liminf_{K \to \infty} \frac{1}{L} \sum_{k=0}^{L-1} \ell(X^{cl}(k),U^{cl}(k)) \\
        &\leq \limsup_{K \to \infty} \frac{1}{L} \sum_{k=0}^{L-1} \ell(X^{cl}(k),U^{cl}(k)) \leq \ell(\mathbf{X}^s,\mathbf{U}^s) + C_f.
    \end{split}
    \end{equation*}
    Moreover, if $C_f = 0$ holds in Assumption~\ref{ass:terminalCost}, then the closed-loop solution has optimal averaged performance, i.e.,
    \begin{equation*}
        \lim_{L \to \infty} \frac{1}{L} \sum_{k=0}^{L-1} \ell(X^{cl}(k),U^{cl}(k)) = \ell(\mathbf{X}^s,\mathbf{U}^s).
    \end{equation*}
\end{cor}
\begin{proof}
    Follows by Theorem~\ref{thm:lowerBoundPerformance} and Theorem~\ref{thm:upperBoundPerformance}.
\end{proof}

\section{MOMENT-BASED REFORMULATION FOR LINEAR-QUADRATIC PROBLEMS WITH GAUSSIAN NOISE} \label{sec:GaussianSetting}

Although our theory applies to general costs and disturbances, the open-loop problems~\eqref{eq:openloopIfSMPC} are in general hard to solve.
In this section, we make some simplifications that enable us to obtain an implementable version of Algorithm~\ref{alg:ifStochMPC}, which only uses information about the expectation and covariances of the appearing quantities.

We consider linear-quadratic stage costs of the form 
\begin{equation} \label{eq:stagecostLQP}
    \ell(X,U) = \E[X^\top  Q X + U^\top  R U] = \E[ \Vert X \Vert_Q^2 + \Vert U \Vert_R^2 ],
\end{equation}
where $Q \in \R^{n \times n}$ is symmetric, positive semi-definite, and $R \in \R^{l \times l}$ is symmetric and positive definite.
Furthermore, we consider a terminal penalty 
\begin{equation} \label{eq:terminalCostLQP}
    F(X) = \E[X^\top  P X] = \E[ \Vert X \Vert^2_P],
\end{equation}
where $P$ is the solution of the Lyapunov equation
\begin{equation} \label{eq:terminalP}
    (A+BK)^\top  P (A+BK) - P = - (Q + KRK^\top ).
\end{equation}
Then, for a given measurement $x_j \in \R^n$ we can evaluate the cost in~\eqref{eq:openloopIfSMPC} as
\begin{equation}
\begin{split}
    \sum_{k=0}^{N-1} & \Big( \mu_X(k)^\top  (Q+K^\top RK) \mu_X(k) + v_k^\top  R v_k \\
    & + \trace((Q+K^\top RK) \Sigma_X(k)) \Big) \\
    + & \mu_X(N)^\top  P \mu_X(N) + \trace(P \Sigma_X(N))
\end{split}
\end{equation}
with
\begin{align*}
    \mu_X(k+1) &= (A+BK)\mu_x(k) + Bv_k + E\mu_w, \\
    \Sigma_X(k) &= (A+BK) \Sigma_X(k) (A+BK)^\top  + \Sigma_W, \\
\end{align*}
Here, $\mu_X(k) = \E[X(k)]$, $\Sigma_X(k) = \text{Cov}(X(k))$, $\mu_W=\E[W(k)]$, and $\Sigma_W = \text{Cov}(W(k))$, and the initial condition is $(\mu_X(0),\Sigma_X(0)) = (x_j,0)$.

Moreover, we assume that the disturbance follows a Gaussian distribution, i.e., $ W(k) \sim \Normal(\mu_W,\Sigma_W)$ holds for all $j \in \N_0$,
and consider the case that the risk measure $\rho(Y)$ defining the risk-averse constraints is one of the following mappings:

\begin{enumerate}[(i)]
    \item The expected value
    \begin{equation} \label{eq:Exp}
         \E[Y] = \E^{\Prob}[Y] := \int_{\Omega} Y d\Prob.
    \end{equation}
    \item The value-at-risk 
    \begin{equation} \label{eq:VaR}
        \text{VaR}_{1-\alpha}(Y):=\inf_{t\in\R}\{t:\Prob(Y \leq t)\geq 1-\alpha\}.
    \end{equation}
    \item The conditional value-at-risk
    \begin{equation} \label{eq:CVaR}
        \text{CVaR}_{1-\alpha}(Y) := \frac{1}{\alpha}\int_0^{\alpha} \text{VaR}_{1-\gamma}(Y)d\gamma.
    \end{equation}
    \item The entropic value-at-risk
    \begin{equation} \label{eq:EVaR}
        \text{EVaR}_{1-\alpha}(Y):=\inf_{z>0}\{z^{-1}\ln(M_Y(z)/\alpha)\},
    \end{equation}
    where $M_Y(z)$ denotes the moment generating function of $Y$ at $z$, which we consider to exists.
\end{enumerate}

\begin{rem}
    We want to emphasize that the constraint $\text{VaR}_{1-\alpha}(Y) \leq d$ is equivalent to $\Prob(Y \leq d) \geq 1 - \alpha$.
    Hence, our setting does also include chance constraint formulations as a special case and thus it can be seen as a extension of \cite{Hewing2020} in terms of the class of constraints.
\end{rem}

Since $W(k)$ has a Gaussian distribution, we can conclude that for an initial value $E_0 = X_0 - E[X_0] \sim \Normal(0,\Sigma_{E_0})$ the random variable $E$ from \eqref{eq:dynamicsE} has a zero-mean Gaussian distribution for all times $k \in \N_0$, i.e, $E(k) \sim \Normal(0,\Sigma_E(k))$ with covariance 
\begin{equation*}
\begin{split}
    \Sigma_E(k+1) &= (A+BK) \Sigma_E(k) (A+BK)^\top  + \Sigma_W, \\
    \Sigma_E(0) &= \Sigma_{E_0}.
\end{split}
\end{equation*}
Using this observation we can evaluate $\rho(c_i^\top  E(j))$ and $\rho(d_i^\top KE(j))$ exactly, since for a random variable $Y \sim \Normal(\mu_Y,\sigma_Y^2)$ with mean $\mu_Y \in \R$ and variance $\sigma_Y^2 \in \R_0^+$ the risk measures \eqref{eq:Exp} -- \eqref{eq:EVaR} can be written as
\begin{equation}
    \rho(Y) = \mu_Y + \sigma_Y R(\alpha), \quad \alpha \in (0,1)
\end{equation}
with
\begin{equation} \label{eq:riskNormal}
    R(\alpha) =
    \begin{cases}
        0 & \text{if } \rho(Y) = \E[Y] \\
        \Phi^{-1}(1-\alpha) & \text{if } \rho(Y) = \text{VaR}(Y) \\
        \frac{\varphi(\Phi^{-1}(1-\alpha))}{\alpha} & \text{if } \rho(Y) = \text{CVaR}(Y) \\
        \sqrt{-2\ln(\alpha)} & \text{if } \rho(Y) = \text{EVaR}(Y)
    \end{cases}
\end{equation}
where $\varphi(x)=\frac{1}{\sqrt{2\pi}}e^{-\frac{x^2}{2}}$ is the standard normal probability density function and $\Phi(x)$ is the standard normal cumulative distribution function.

To construct the terminal set $\mathbb{Z}_f$ let us now assume that $\Sigma_{E_0} \preceq \Sigma_E^s$ holds, where $\Sigma_E^s$ is the solution of the Lyapunov equation
\begin{equation}
    \Sigma_E^s = (A+BK) \Sigma_E^s (A+BK) + \Sigma_W
\end{equation}
Then we can conclude that $\Sigma_{E}(j) \preceq \Sigma_E^s$ holds for all $j \in \N_0$ and hence,
\begin{align*}
    \sup_{j \in \N_0} \rho(c_i^\top E(j)) &\leq c_i^\top  \Sigma_E^s c_i R(\alpha),~ i=1,\ldots,m_x \\
    \sup_{j \in \N_0} \rho(d_i^\top KE(j)) &\leq d_i^\top K^\top  \Sigma_E^s K d_i R(\alpha),~ i=1,\ldots,m_u .
\end{align*}
Thus, assuming that 
\begin{align*}
    p_i &\geq c_i^\top  \Sigma_E^s c_i R(\alpha), \quad i=1,\ldots,m_x \\
    q_i &\geq d_i^\top K^\top  \Sigma_E^s K d_i R(\alpha), \quad i=1,\ldots,m_x
\end{align*}
and $0 \in \mathbb{V}$ holds, the terminal set $\mathbb{Z}_f = \{ 0 \}$
satisfies the conditions of Assumption~\ref{ass:terminalSet} with $v_f=0$ since $(z^s,v^s) = (0,0)$ is an equilibrium of the dynamic~\eqref{eq:dynamicsz}
for $\E[W(k)]=0$. 

The resulting moment-based open-loop problem can be summarized as 
\begin{equation} \label{eq:openloopIfSMPCmoments}
    \begin{split}
        \min_{\mathbf{v}} \sum_{k=0}^{N-1} & \Big( \mu_X(k)^\top  (Q+K^\top RK) \mu_X(k) + v_k^\top  R v_k \\
        & + \trace((Q+K^\top RK) \Sigma_X(k)) \Big) \\
        + & \mu_X(N)^\top  P \mu_X(N) + \trace(P \Sigma_X(N)) \\
        s.t.~\mu_X(k+1) &= (A+BK) \mu_X(k) + Bv_k  \\
        \Sigma_X(k+1) &= (A+BK) \Sigma_X(k) (A+BK)^\top  + \Sigma_W \\
        z(k+1) &= (A+BK)z(k) + Bv_k  \\
        E(k+1) &= (A+BK) \Sigma_E(k) (A+BK)^\top  + \Sigma_W \\
        U(k) &= KX(k) + v_k , ~ v_k \in \mathbb{V}, ~ z(N) = 0 \\
        \mu_X(0) =& x_j,~ \Sigma_X(0) = 0, ~ z(0) = z_j, ~\Sigma_E(0) = \Sigma_{E_j}\\
        c_i^\top z(k) &\leq p_i - c_i^\top  \Sigma_E(k) c_i R(\alpha), \quad i=1,\ldots,m_x  \\
        d_i^\top (Kz(k) &+ v_k) \leq q_i - d_i^\top K^\top  \Sigma_E(k) K d_i R(\alpha) \\
         i&=1,\ldots,m_u .
    \end{split}
\end{equation}
and the corresponding MPC scheme is given in Algorithm~\ref{alg:ifStochMPCmoments}.

Since the terminal set $\mathbb{Z}_f$ satisfies Assumption~\ref{ass:terminalSet} we directly get recursive feasibility by Theorem~\ref{thm:recursiveFeasibility} and closed-loop constraint satisfaction by Theorem~\ref{thm:constraints}.
However, to ensure also the performance bounds from Section~\ref{sec:Performance} we need to show that stochastic dissipativity holds and Assumption~\ref{ass:terminalCost} is satisfied for the terminal cost from~\eqref{eq:terminalCostLQP}.
For the simplified setting of this section this is shown by the following theorem and lemma.

\begin{thm} \label{thm:DissipativityLQP}
    Let $P^*$ be the solution of the discrete-time algebraic Riccati equation
    \begin{equation} \label{eq:DARE}
         P^* = A^\top  P^* A + Q A^\top  P^* B(R + B^\top  P^* B)^{-1} B^\top  P^* A
    \end{equation}
    and set $K^* := - (R + B^\top  P^* B)^{-1} B^\top  P^* A$.
    Furthermore, let $\Sigma_X^s$ be the solution of
    \begin{align*}
        \Sigma_X^s = (A+BK^*) \Sigma_X^s (A+BK^*)^\top  + \Sigma_W
    \end{align*}
    and assume that 
    \begin{equation} \label{eq:StationaryPairAdmissable}
    \begin{split}
        p_i &\geq c_i^\top  \Sigma_X^s c_i R(\alpha), ~ i=1,\ldots,m_x\\
        q_i &\geq d_i^\top (K^*)^\top  \Sigma_X^s (K^*) d_i R(\alpha), ~ i=1,\ldots,m_u
    \end{split}
    \end{equation}
    holds.
    Then there exits a stationary pair $(\mathbf{X}^s,\mathbf{U}^s)$ with $U^s(k) = K^*X^s(k)$, $X \sim \Normal(0,\Sigma_X^s)$ for all $k \in \N_0$, and 
    \begin{equation}
        \ell(\mathbf{X}^s,\mathbf{U}^s) =  \trace((Q+K^\top RK) \Sigma_X^s) = \trace(P^* \Sigma_W)
    \end{equation}
    such that the stochastic optimal control problem~\eqref{eq:stochOCP} under the simplifications of this section is stochastically dissipative.
\end{thm}
\begin{proof}
    By \cite[Theorem~3.11]{Schiessl2025b} we can conclude that the stochastic optimal control problem with the linear-quadratic structure of this section and without constraints is stochastically dissipative at $(\mathbf{X}^s,\mathbf{U}^s)$.
    However, since the stationary pair $(\mathbf{X}^s,\mathbf{U}^s)$ satisfies the constraints due to the assumption from \eqref{eq:StationaryPairAdmissable}, the constrained problem is also stochastically dissipative at $(\mathbf{X}^s,\mathbf{U}^s)$.
\end{proof}

\begin{lem}
    The terminal cost $F(X) = \E[X^\top  P X]$ with $P$ from equation~\eqref{eq:terminalP} satisfies Assumption~\ref{ass:terminalCost} for $v_f = 0$ and $(\mathbf{X}^s,\mathbf{U}^s)$ from Theorem~\ref{thm:DissipativityLQP}.
\end{lem}
\begin{proof}
    Using the relation from \eqref{eq:terminalP} and that $X$ and $W(N)$ are independent, we obtain for all $X \in \RR(\Omega,\R^n)$ that 
    \begin{align*}
        &F((A+BK)X + W(N)) \\
        =& \E[\Vert (A+BK)X + W(N) \Vert_P^2] \\
        =& \E[\Vert X \Vert_{(A+BK)P(A+BK)^\top }^2] + \E[ \Vert W(N) \Vert_P^2] \\
        =& \E[\Vert X \Vert_{P-(Q + KRK^\top )}^2]  + \E[ \Vert W(N) \Vert_P^2] \\
        =& \E[\Vert X \Vert_{P}^2] - (\E[\Vert X \Vert_{Q}^2] + \E[\Vert KX \Vert_{R}^2]) + \E[ \Vert W(N) \Vert_P^2] \\
        =& F(X) - \ell(X,KX) + \ell(\mathbf{X}^s,\mathbf{U}^s) + C_f
    \end{align*}
    with $C_f = \trace((P-P^*)\Sigma_W) \geq 0$.
    Note that $C_f \geq 0$ must hold 
    since stochastic dissipativity implies that the stationary pair $(\mathbf{X}^s,\mathbf{U}^s)$ has optimal stationary cost, cf. \cite[Theorem~5.2]{Schiessl2025b}.
\end{proof}

Based on these two results the following corollary summarizes the implications from Corollary~\ref{cor:performance} for the linear-quadratic Gaussian setting of this section.

\begin{cor} \label{cor:performanceMoments}
    Let the simplifications of this section and the assumptions of Theorem~\ref{thm:DissipativityLQP} hold.
    Then, we obtain
    \begin{equation*}
    \begin{split}
        \trace(P^* &\Sigma_W) \leq \liminf_{L \to \infty} \frac{1}{L} \sum_{k=0}^{L-1} \ell(X^{cl}(k),U^{cl}(k)) \\
        &\leq \limsup_{K \to \infty} \frac{1}{L} \sum_{k=0}^{L-1} \ell(X^{cl}(k),U^{cl}(k)) \leq \trace(P \Sigma_W).
    \end{split}
    \end{equation*}
    Moreover, if we choose $K=K^*$ as the fixed linear feedback we get
    \begin{equation*}
    \begin{split}
        \lim_{L \to \infty} \frac{1}{L} \sum_{k=0}^{L-1} \ell(X^{cl}(k),U^{cl}(k)) = \trace(P^* \Sigma_W).
    \end{split}
    \end{equation*}
\end{cor}

\begin{algorithm}
    \caption{Moment-based indirect feedback SMPC}
    \label{alg:ifStochMPCmoments}
    \begin{algorithmic}
    \State \textbf{Input:} Fixed stabilizing feedback $K \in \R^{l \times n}$, feasible initial state $X_0 \sim \Normal(\mu_{X_0},\Sigma_{X_0})$ with $\Sigma_{X_0} \preceq \Sigma_E^s$.
        \State Measure the state $x_0 = X_0(\omega)$, set $z_0 = \mu_{X_0}$, $\Sigma_{E_0} = \Sigma_{X_0}$, $X^{cl}(0,\omega)=x_0$, $Z^{cl}(0,\omega) = z_0$.
        \For{$j=0,1,\ldots$}
            \State 1.) Solve the stochastic optimal control problem \eqref{eq:openloopIfSMPCmoments} and obtain the solution $\mathbf{v}^* := (v_0^*,\ldots,v_{N-1}^*).$
            \State 2.) Compute $\Sigma_{E_{j+1}} = (A+BK)\Sigma_{E_j}(A+BK)^\top$,  predict 
            $z_{j+1} = (A+BK)z_j + Bv_0^*$,
            and set $Z^{cl}(j+1,\omega) = z_{j+1}$.
            \State 3.) Set $V^{cl}(j,\omega) = v_0^*$, apply the feedback 
            $U^{cl}(j,\omega) = Kx_j + v_0^*$
            to system \eqref{eq:StochSystem} and measure the next state $x_{j+1} = X^{cl}(j+1,\omega)$.
        \EndFor
    \end{algorithmic}
\end{algorithm}

\section{NUMERICAL EXAMPLE} \label{sec:Numerics}
In this section we will illustrate our findings by  a DC-DC-converter regulation problem, which was already used for case studies in stochastic MPC in \cite{cannon2010stochastic,schluter2022stochastic,schluter2023stochastic}
The corresponding dynamics are of the form~\eqref{eq:StochSystem}, where
\begin{equation}
    A = \left(
    \begin{array}{cc}
       1  &  0.0075 \\
       -0.143  & 0.996
    \end{array}
    \right), \quad
    B = \left(
    \begin{array}{c}
         4.798 \\
         0.115
    \end{array}
    \right),
\end{equation}
and $W(k) \sim \Normal(0,\Sigma_W)$ with $\Sigma_W = 0.1 I_2 \in \R^{2 \times 2}$.
Additionally, we consider quadratic costs of the form \eqref{eq:stagecostLQP} with $Q = \text{diag}(1,10)$ and $R=5$ and impose a single risk-averse constraint on the first component given by 
\begin{equation} \label{eq:constr_expl}
    \rho(X_1(k)) \leq 2
\end{equation}
where $\rho(Y)$ denotes one of the risk measures from equation ~\eqref{eq:Exp} -- \eqref{eq:EVaR} with $1-\alpha = 0.6$.

To obtain the closed-loop quantities, we generated 15\,000 samples using Algorithm~\ref{alg:ifStochMPCmoments} with $K = K^*$ from Theorem~\ref{thm:DissipativityLQP} and a deterministic initial value $X_0 = (1.8, 1.5)^\top$.
Figure~\ref{fig:constraints} shows the evolution of $\mathbb{E}[X^{cl}_1(k)]$, $\text{VaR}_{0.6}(X^{cl}_1(k))$, $\text{CVaR}_{0.6}(X^{cl}_1(k))$, and $\text{EVaR}_{0.6}(X^{cl}_1(k))$ for different choices of $\rho \in \{\mathbb{E}$, $\text{VaR}_{0.6}$, $\text{CVaR}_{0.6}$, $\text{EVaR}_{0.6}\}$.
We clearly observe that the constraints are always satisfied, as predicted by Theorem~\ref{thm:constraints}.

Furthermore, for all $\alpha \in (0,1)$, $Z \in \RR(\Omega,\R)$ it holds that
\begin{equation*} 
    \mathbb{E}(Z) \leq \text{VaR}_{1-\alpha}(Z) \leq \text{CVaR}_{1-\alpha}(Z) \leq \text{EVaR}_{1-\alpha}(Z).
\end{equation*}
Consequently, the restrictiveness of the constraints follows the same ordering, which is also observable in Figure~\ref{fig:constraints}.

\begin{figure}
    \centering
    \includegraphics[width=0.9\linewidth, trim={0.1cm 1.05cm 0.1cm 0.93cm}, clip]{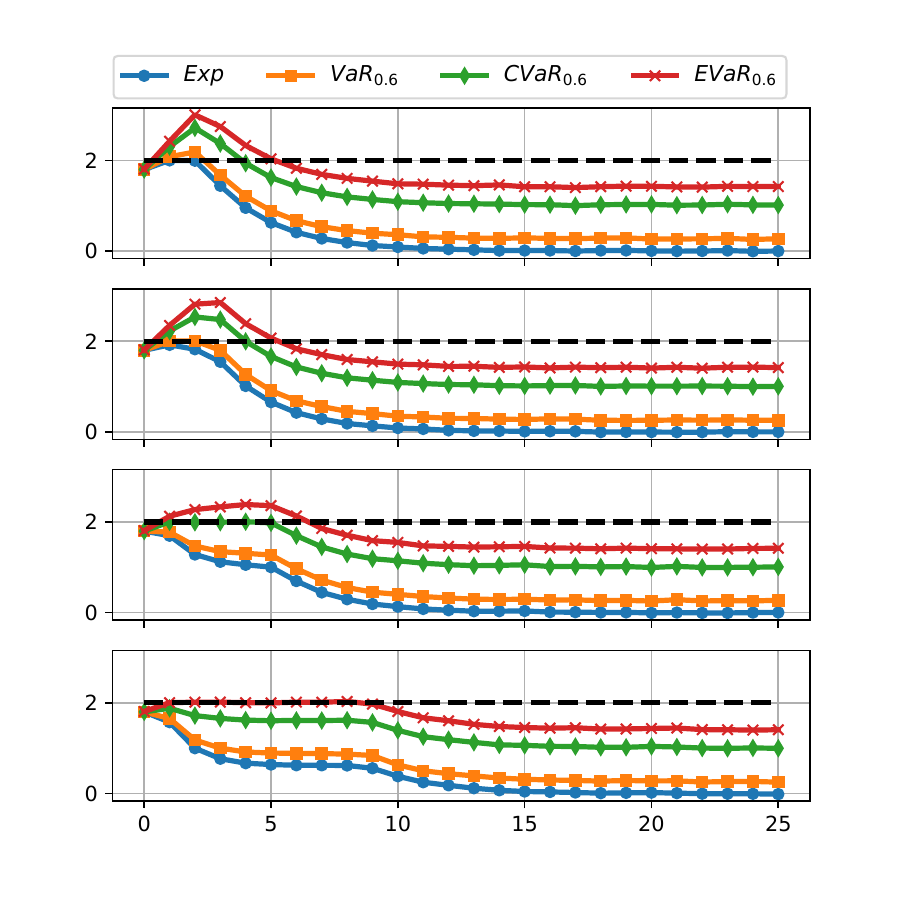}
    \caption{Evolution of $\E[X^{cl}_1(k)]$ (blue), $\text{VaR}_{0.6}(X^{cl}_1(k))$ (orange), $\text{CVaR}_{0.6}(X^{cl}_1(k))$ (green), and $\text{EVaR}_{0.6}(X^{cl}_1(k))$ (red) as well as the upper constraint bound $p = 2$ (dashed black). The subplots (top to bottom) correspond to the constraints~\eqref{eq:constr_expl} with $\rho = \mathbb{E}$, $\rho = \text{VaR}_{0.6}$, $\rho = \text{CVaR}_{0.6}$, and $\rho = \text{EVaR}_{0.6}$, respectively.}
    \label{fig:constraints}
\end{figure}

To compare the performance for different choices of $K$ in Algorithm~\ref{alg:ifStochMPCmoments}, we constructed a second stabilizing feedback
by solving the algebraic Ricatti equation~\eqref{eq:DARE} for $\tilde{Q} := 0.01 Q$ and $\tilde{R} := 200 R$.
We then ran Algorithm~\ref{alg:ifStochMPCmoments} again using these different choices of $K$, where the constraints in \eqref{eq:constr_expl} were defined using the conditional value-at-risk $\text{CVaR}_{0.6}$.

As Corollary~\ref{cor:performanceMoments} indicates, for $K = K^*$ the averaged performance should converge to the optimal stationary cost, as can be seen in Figure~\ref{fig:performance}.
However, for $K = \tilde{K}$, Corollary~\ref{cor:performanceMoments} only guarantees that the averaged performance satisfies a suboptimal bound; convergence to the optimal stationary cost is therefore not ensured.
Figure~\ref{fig:performance} shows that the averaged closed-loop performance indeed converges to a value deviating from the optimal stationary cost, demonstrating that the choice of $K$ affects the asymptotic performance of the closed-loop solution not only theoretically but also in practice.

\begin{figure}
    \centering
    \includegraphics[width=0.7\linewidth, trim={0.1cm 0.4cm 0.1cm 0.2cm}, clip]{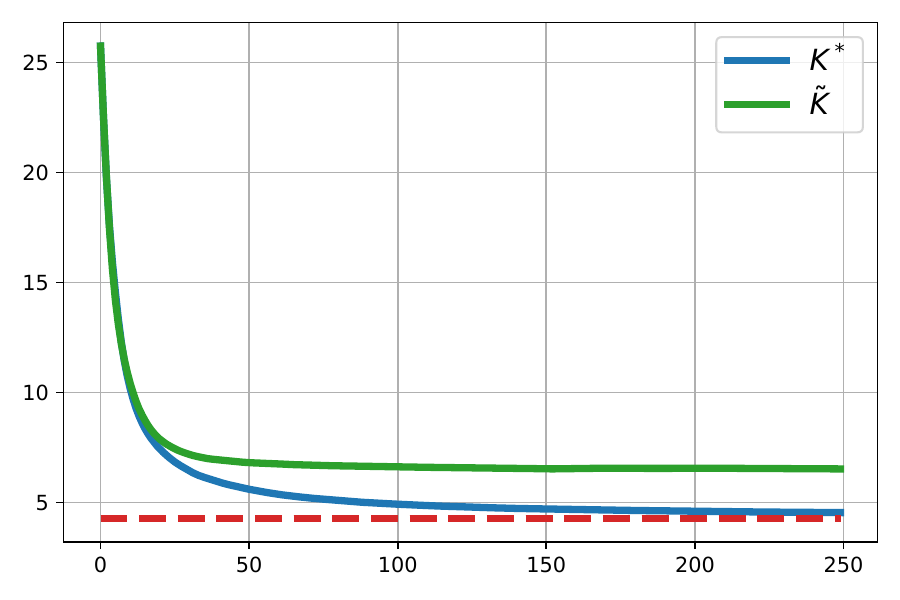}
    \caption{Averaged closed-loop performance of Algorithm~\ref{alg:ifStochMPCmoments} for $K = K^*$ (blue) and $K=\tilde{K}$ (green) if choosing $\rho=\text{VaR}_{0.6}$ in~\eqref{eq:constr_expl} as well as the optimal stationary cost (red dashed).}
    \label{fig:performance}
\end{figure}

\section{CONCLUSION} \label{sec:Conclusion}

We presented an indirect feedback approach for stochastic MPC with linear systems and general risk-averse constraints defined via risk measures.
For this algorithm, we derived near-optimal performance bounds for general cost functions.
Future research should focus on developing efficient implementations of Algorithm~\ref{alg:ifStochMPC} without the simplifications introduced in Section~\ref{sec:GaussianSetting}, constructing suitable terminal costs for the non-quadratic case, or extending the presented results to nonlinear systems, as in \cite{kohler2025predictive} for the original chance-constrained algorithm.

{\bibliographystyle{abbrv}
  \bibliography{references} 
}

\end{document}